\documentclass[12pt,twoside]{amsart}
\usepackage{amsmath}
\usepackage{amssymb}
\usepackage{amsthm}
\usepackage{latexsym}
\usepackage{amscd}
\usepackage{xypic}
\xyoption{curve}
\usepackage{ifthen} 
\usepackage{hyperref} 
\usepackage{graphicx}
\usepackage{enumerate}
\usepackage{enumitem}

 \def\cocoa{{\hbox{\rm C\kern-.13em o\kern-.07em C\kern-.13em o\kern-.15em A}}}

\usepackage{xcolor}

\newtheorem{theorem}{Theorem}[section]
\newtheorem{lemma}[theorem]{Lemma}
\newtheorem{proposition}[theorem]{Proposition}
\newtheorem{corollary}[theorem]{Corollary}

\theoremstyle{definition}
\newtheorem{remark}[theorem]{Remark}

\newtheorem{example}[theorem]{Example}

\newcommand {\Hom}{\mathcal{H}\kern -0.25ex{\mathit om}}
\newcommand {\Ext}{\mathcal{E}\kern -0.25ex{\mathit xt}}

\newcommand {\rk}{\mathrm{rk}}

\newcommand {\ext}{\mathrm{Ext}}
\newcommand {\Hilb}{\mathcal{H}\kern -0.25ex{\mathit ilb\/}}

\newcommand {\Cliff}{\mathrm{Cliff}}

\newcommand {\cK}{\mathcal{K}}
\newcommand {\cA}{\mathcal{A}}
\newcommand {\cB}{\mathcal{B}}

\newcommand {\bZ}{\mathbb{Z}}

\newcommand {\bP}{\mathbb{P}}

\newcommand {\bF}{\mathbb{F}}

\newcommand{\cU}{{\mathcal U}}

\newcommand{\cE}{{\mathcal E}}
\newcommand{\cF}{{\mathcal F}}
\newcommand{\cM}{{\mathcal M}}

\newcommand{\cO}{{\mathcal O}}
\newcommand{\cG}{{\mathcal G}}
\newcommand{\cI}{{\mathcal I}}

\newcommand{\Bl}{\operatorname{Bl}}
\newcommand{\Pic}{\operatorname{Pic}}

\def\p#1{{\bP^{#1}}}

\def\ga#1{{{\accent"12 #1}}}
 
\title[Special Ulrich bundles on non--special surfaces with $p_g=q=0$]{Special Ulrich bundles on \\ non--special surfaces with $p_g=q=0$}

\subjclass[2010]{Primary 14J60; Secondary 14J26, 14J27, 14J28, 14J29}

\keywords{Vector bundle, Ulrich bundle.}

\author[Gianfranco Casnati]{Gianfranco Casnati}
\thanks{The author is a member of GNSAGA group of INdAM and is supported by the framework of PRIN 2015 \lq Geometry of Algebraic Varieties\rq, cofinanced by MIUR}

\begin{document}

\begin{abstract}
Let $S$ be a surface with $p_g(S)=q(S)=0$ and endowed with a very ample line bundle $\mathcal O_S(h)$ such that $h^1\big(S,\mathcal O_S(h)\big)=0$. We show that $S$ supports special (often stable) Ulrich bundles of rank $2$, extending a recent result by A. Beauville. Moreover, we show that such an $S$ supports families of dimension $p$ of pairwise non--isomorphic, indecomposable, Ulrich bundles for arbitrary large $p$ except for very few cases. We also show that the same is true for linearly normal non--special surface in $\p4$ of degree at least $4$, Enriques surface and anticanonical rational surface.

\end{abstract}

\maketitle

\section{Introduction and Notation}
Throughout the whole paper we will work on an algebraically closed field $k$ of characteristic $0$ and $\p N$ will denote the projective space over $k$ of dimension $N$. The word surface will always denote a projective smooth connected surface.

Each variety $X\subseteq\p N$ is naturally endowed with the very ample line bundle $\cO_X(h):=\cO_{\p N}(1)\otimes\cO_X$. In order to understand the geometry of the embedded variety $X$, it could be helpful to deal with the vector bundles supported on $X$. From a cohomological point of view, the simplest bundles $\cF$ on $X$ are the ones which are {\sl Ulrich with respect to $\cO_X(h)$}, i.e. such that 
$$
h^i\big(X,\cF(-ih)\big)=h^j\big(X,\cF(-(j+1)h)\big)=0
$$
for each $i>0$ and $j<\dim(X)$. 

The existence of Ulrich bundles on a fixed polarised variety is a challenging problem raised by D. Eisenbud and F.O. Schreyer in \cite{E--S--W}. It is wide open, though many scattered results are known: without any claim of completeness we recall \cite{A--C--MR}, \cite{A--F--O}, \cite{Bea}, \cite{Bea1}, \cite{Bea2}, \cite{B--N},  \cite{C--H1}, \cite{C--H2}, \cite{C--K--M2}, \cite{PL--T}.

Now let $S\subseteq\p N$ be a surface and set $p_g(S)=h^2\big(S,\cO_S\big)$, $q(S)=h^1\big(S,\cO_S\big)$. It is of particular interest to understand the existence of Ulrich bundles of rank $2$ which are {\sl special} in the sense of \cite{E--S--W}, i.e. with first Chern class $3h+K_S$, $K_S$ being the canonical class on $S$. Notice that not all the Ulrich bundles are special (for some explicit examples see e.g. \cite{C--K--M1}, \cite{Cs} and \cite{C--G}).

The first result proved in this paper is the following. Recall that $S$ or $\cO_S(h)$ are called {\sl special} if $h^1\big(S,\cO_S(h)\big)\ne0$, {\sl non--special} otherwise. 

\begin{theorem}
\label{tExistence}
Let $S$ be a surface with $p_g(S)=q(S)=0$ and endowed with a very ample non--special line bundle $\cO_S(h)$.

For each general set $Z\subseteq S$ of $h^0\big(S,\cO_S(h)\big)+1$ points, there is a special Ulrich bundle $\cE$ with respect to $\cO_S(h)$ of rank $2$ fitting into the exact sequence
\begin{equation}
\label{seqUlrich}
0\longrightarrow\cO_S(h+K_S)\longrightarrow\cE\longrightarrow\cI_{Z\vert S}(2h)\longrightarrow0.
\end{equation}
\end{theorem}

A similar result has been proved by A. Beauville with the slightly stronger assumption that $\vert h-K_S\vert$ contains an irreducible curve instead of the vanishing $h^1\big(S,\cO_S(h)\big)=0$.

Notice that the bundle in the above theorem could be a direct sum of line bundles: e.g. if  $S:=\p2$ and $\cO_S(h):=\cO_{\p2}(1)$, then $\cE\cong\cO_{\p2}^{\oplus2}$ in Sequence \eqref{seqUlrich}. Thus, it is  interesting to understand when the construction above returns an {\sl indecomposable} bundle, i.e. a bundle which is not the direct sum of bundles of lower degree.

A condition forcing the indecomposability of $\cE$ is its stability. Recall that an Ulrich bundle $\cF$ on the surface $S$ endowed with the very ample line bundle $\cO_S(h)$ is called {\sl stable} if  $c_1(\cG)h/\rk(\cG)<c_1(\cF)h/\rk(\cF)$ for each proper subbundle $\cG\subseteq\cF$ (see Section \ref{sStability} for further comments and result on this notion). It is not clear a priori whether the bundles constructed in Theorem \ref{tExistence} are stable. 
In Section \ref{sStability} we prove the following result.

\begin{theorem}
\label{tStable}
Let $S$ be a surface with $p_g(S)=q(S)=0$ and endowed with a very ample non--special line bundle $\cO_S(h)$.

Let $\cE$ be the bundle constructed in Theorem \ref{tExistence} from a set $Z\subseteq S$ of $h^0\big(S,\cO_S(h)\big)+1$ points:
\begin{enumerate}
\item $\cE$ is never stable for each choice of $Z$ if either $\cO_S(h)$ embeds $S$ as a rational scroll or $S\cong\p2$ and $\cO_S(h)\cong\cO_{\p2}(1)$.
\item $\cE$ is stable for a general choice of $Z$ in all the remaining cases.
\end{enumerate}
\end{theorem}

In \cite{F--PL}, D. Faenzi and J. Pons--Llopis also discussed about the existence of Ulrich bundles of higher rank on a variety $X$ as a sign of the complexity of the variety itself. For example one could ask if $X$ is of {\sl Ulrich--wild representation type}, i.e. if it supports families of dimension $p$ of pairwise non--isomorphic, indecomposable, Ulrich bundles for arbitrary large $p$. 

Ulrich--wildness (with respect to a suitable line bundle) is known for each anticanonically embedded surface of degree at least $3$ (see e.g. \cite{C--H2}, \cite{C--K--M1}, \cite{MR--PL2}, \cite{PL--T}, \cite{Cs}), all the Segre products but $\p1\times\p1$ (see \cite{C--MR--PL}), determinantal varieties (see \cite{K--MR}), quartic surfaces in $\p3$ and general surfaces in $\p3$ of degree $d=5,6,7,8,9$ (see \cite{Cs1}), rational surfaces $S\subseteq\p4$ which are {\sl aCM}, i.e. projectively normal and such that $h^1\big(S,\cO_S(th)\big)=0$ for each $t\in\bZ$ (see \cite{MR--PL1}).

Let 
$$
\pi(\cO_S(h)):=\frac{h^2+hK_S}2+1
$$
be the {\sl sectional genus} of the polarised surface $S$. The third theorem proved in this paper is the following.

\begin{theorem}
\label{tWild}
Let $S$ be a surface with $p_g(S)=q(S)=0$ and endowed with a very ample non--special line bundle $\cO_S(h)$.

Then $S$ is Ulrich--wild if and only if either $\pi(\cO_S(h))\ge1$, or $\pi(\cO_S(h))=0$ and $h^2\ge5$.
\end{theorem}

As a first by--product of the above theorem we are able to extend the aforementioned results from \cite{MR--PL1} to each non--special surface $S\subseteq\p4$, which is {\sl linearly normal}, i.e. such that $h^0\big(S,\cO_S(h)\big)=5$ according to the definition in \cite{Al}. 

As a second consequence it also allows us to prove that every Enriques surface $S$ is Ulrich--wild with respect to each very ample line bundle, thus extending a result proved in \cite{B--N} under quite restrictive hypothesis on $S$.

In \cite{Kim}, the author deals with the existence of Ulrich bundles on a rational surface $S$ which is {\sl anticanonical}, i.e. such that $\vert-K_S\vert\ne\emptyset$. Using a construction due to Lazarsfeld and Mukai, the author proves in Theorem 1 of \cite{Kim} that $S$ supports an irreducible family of dimension $h^2-K_S^2+5$  of rank $2$ stable special Ulrich bundles under a rather technical condition on the curves in $\vert 3h+K_S\vert$. 

We will see in Section \ref{sAnticanonical} that such a technical condition actually forces the inequality $\pi(\cO_S(h))\ge1$. Thus, thanks to Theorems \ref{tExistence} and \ref{tStable}, the aforementioned Theorem 1 of \cite{Kim} can be generalised as follows.

\begin{theorem}
\label{tAnticanonical}
Let $S$ be an anticanonical rational surface endowed with a very ample line bundle $\cO_S(h)$. 

Then $\cO_S(h)$ is non--special and $S$ supports special Ulrich bundles of rank $2$. If $\pi(\cO_S(h))\ge1$, then $S$ supports stable special Ulrich bundles of rank $2$: their moduli space is irreducible rational and smooth of dimension $h^2-K_S^2+5$. 
\end{theorem}

In Section \ref{sGeneral} we list some general results on Ulrich bundles on polarised surfaces. In Section \ref{sExistence} we prove Theorem \ref{tExistence}. In Section \ref{sStability} we first recall some easy facts about the stability of Ulrich bundles, giving finally the proof of Theorem \ref{tStable}. In Section \ref{sWild} we prove Theorem \ref{tWild} and we show as by--products that both Enriques surfaces, and linearly normal non special rational surfaces in $\p4$ of degree at least $4$ are Ulrich--wild. Finally, in Section \ref{sAnticanonical} we prove Theorem \ref{tAnticanonical}.

\section{General results}
\label{sGeneral}
In general, an Ulrich bundle $\cF$ on a smooth variety $X\subseteq\p N$ collects many interesting properties (see Section 2 of \cite{E--S--W}). The following ones are particularly important.
\begin{itemize}
\item $\cF$ is globally generated and its direct summands are Ulrich as well. 
\item $\cF$ is {\sl initialized}, i.e. $h^0\big(X,\cF(-h)\big)=0$ and $h^0\big(X,\cF\big)\ne0$.
\item $\cF$ is {\sl aCM}, i.e. $h^i\big(X,\cF(th)\big)=0$ for each $i=1,\dots,\dim(X)-1$ and $t\in \bZ$.
\item $\cF$ is {\sl maximally generated}, i.e. $h^0\big(X,\cF(-h)\big)$ attains the maximal possible value for an aCM bundle on $X$, i.e. $\rk(\cF)h^n$.
\end{itemize}

Now we turn our attention to a surface $S$. The Serre duality for $\cF$ is
$$
h^i\big(S,\cF\big)=h^{2-i}\big(S,\cF^\vee(K_S)\big),\qquad i=0,1,2,
$$
and the Riemann--Roch theorem is
\begin{equation}
\label{RRGeneral}
h^0\big(S,\cF\big)+h^{2}\big(S,\cF\big)=h^{1}\big(S,\cF\big)+\rk(\cF)\chi(\cO_S)+\frac{c_1(\cF)(c_1(\cF)-K_S)}2-c_2(\cF),
\end{equation}
where $\chi(\cO_S):=1-q(S)+p_g(S)$.

\begin{proposition}
\label{pUlrich}
Let $S$ be a surface endowed with a very ample line bundle  $\cO_S(h)$.

If $\cE$ is a vector bundle on $S$, then the following assertions are equivalent:
\begin{enumerate}
\item $\cE$ is an Ulrich bundle;
\item $\cE^\vee(3h+K_S)$ is an Ulrich bundle;
\item $\cE$ is an aCM bundle and 
\begin{equation}
\label{eqUlrich}
\begin{gathered}
c_1(\cE)h=\frac{\rk(\cE)}2(3h^2+hK_S),\\ 
c_2(\cE)=\frac{1}2(c_1(\cE)^2-c_1(\cE)K_S)-\rk(\cE)(h^2-\chi(\cO_S));
\end{gathered}
\end{equation}
\item $h^0\big(S,\cE(-h)\big)=h^0\big(S,\cE^\vee(2h+K_S)\big)=0$ and Equalities \eqref{eqUlrich} hold.
\end{enumerate}
\end{proposition}
\begin{proof}
The bundle $\cE$ is Ulrich if and only if the same is true for $\cE^\vee(3h+K_S)$ thanks to the Serre duality, because $\left(\cE^\vee(3h+K_S)\right)^\vee(3h+K_S)\cong\cE$. Thus the assertions (a) and (b) are equivalent.

Let $\cE$, hence $\cE^\vee(3h+K_S)$, be Ulrich: thus it is aCM, whence $h^1\big(S,\cE(-jh)\big)=0$, $j=1,2$. We know that
\begin{gather*}
h^0\big(S,\cE(-2h)\big)\le h^0\big(S,\cE(-h)\big)=0,\\
h^2\big(S,\cE(-h)\big)=h^0\big(S,\cE^\vee(h+K_S)\big)\le h^0\big(S,\cE^\vee(2h+K_S)\big)=h^2\big(S,\cE(-2h)\big)=0,
\end{gather*}
by definition, thus $\chi(\cE(-h))=\chi(\cE(-2h))=0$. These vanishings and Formula \eqref{RRGeneral} finally yield Equalities \eqref{eqUlrich}. We conclude that the assertion (b), implies (c).

Assume that $\cE$ is aCM and Equalities \eqref{eqUlrich} hold. By using Formula \eqref{RRGeneral}, we have $\chi(\cE(-h))=\chi(\cE(-2h))=0$. Since $h^1\big(S,\cE(th)\big)=0$ for $\in\bZ$, it follows $h^0\big(S,\cE(-h)\big)=h^2\big(S,\cE(-2h)\big)=0$, i.e. the assertion (c) implies (d). 

Arguing similarly, the vanishings $h^0\big(S,\cE(-h)\big)=h^0\big(S,\cE^\vee(2h+K_S)\big)=0$ combined with Equalities \eqref{eqUlrich} and Formula \eqref{RRGeneral} for $\chi(\cE(-h))$ and $\chi(\cE(-2h))$ easily imply $h^i\big(S,\cE(-jh)\big)=0$ for $i=0,1,2$ and $j=1,2$. Hence $\cE$ is Ulrich.
\end{proof}

The following corollary is an immediate consequence of the above characterization.

\begin{corollary}
\label{cUlrichLine}
Let $S$ be a surface endowed with a very ample line bundle  $\cO_S(h)$.

If $\cO_S(D)$ is a line bundle on $S$, then the following assertions are equivalent:
\begin{enumerate}
\item $\cO_S(D)$ is an Ulrich bundle;
\item $\cO_S(3h+K_S-D)$ is an Ulrich bundle;
\item $\cO_S(D)$ is an aCM bundle and 
\begin{equation}
\label{eqLineBundle}
D^2=2(h^2-\chi(\cO_S))+DK_S,\qquad Dh=\frac12(3h^2+hK_S);
\end{equation}
\item $h^0\big(S,\cO_S(D-h)\big)=h^0\big(S,\cO_S(2h+K_S-D)\big)=0$ and Equalities \eqref{eqLineBundle} hold.
\end{enumerate}
\end{corollary}

\begin{example}
\label{eRational}
Let $S$ be the Hirzebruch surface $\Bbb F_e$ where $e$ is a non--negative integer. 

The group $\Pic(S)$ is freely generated by the classes $\xi$ and $f$ satisfying $\xi^2=-e$, $\xi f=1$, $f^2=0$, $K_S=-2\xi-(e+2)f$. If $\cO_S(h)\cong\cO_S(a\xi+bf)$ is very ample, then Ulrich line bundles on $S$ have been classified in \cite{A--C--MR}, Theorem 2.1 when $e>0$ (see also \cite{MR}). In particular an Ulrich line bundle exist on $S$ if and only if $a=1$.

Let us examine the remaining case $e=0$. In this case $S\cong\p1\times\p1$ and we can assume $1\le a\le b$. Let $\cO_S(D)\cong\cO_S(u\xi+vf)$ be Ulrich. Equalities \eqref{eqLineBundle} yield the identities
$$
(u+1)(v+1)=2ab,\qquad a(v+1)+b(u+1)=3ab.
$$
Simple computations thus show that $\cO_S(D)$ is either $\mathcal L:=\cO_S((a-1)\xi+(2b-1)f)$ or $\mathcal M:=\cO_S((2a-1)\xi+(b-1)f)$. It is immediate to check that both these line bundles satisfy the necessary vanishing conditions on the cohomology. Notice that if $a=b=1$ these are exactly the well--known spinor bundles.

Let $S\cong\p2$ and $\cO_S(h)\cong\cO_{\p2}(\lambda)$ where $\lambda\ge1$. Argueing as above one can check that an Ulrich line bundle exist on $S$ if and only if $\lambda=1$ and, in this case, it coincides with $\cO_S$.
\end{example}

Recall that a rank $2$ Ulrich bundle $\cE$ on $S$ is special if $c_1(\cE)=3h+K_S$. 

\begin{corollary}
\label{cUlrichSpecial}
Let $S$ be a surface endowed with a very ample line bundle  $\cO_S(h)$.

If $\cE$ is a vector bundle of rank $2$ on $S$, then the following assertions are equivalent:
\begin{enumerate}
\item $\cE$ is a special Ulrich bundle;
\item $\cE$ is initialized and 
\begin{equation}
\label{eqUlrichSpecial}
c_1(\cE)=3h+K_S,\qquad c_2(\cE)=\frac{1}2(5h^2+3hK_S)+2\chi(\cO_S).
\end{equation}
\end{enumerate}
\end{corollary}
\begin{proof}
The equality $h^0\big(S,\cE(-h)\big)=h^0\big(S,\cE^\vee(2h+K_S)\big)$ is trivially true, thus the statement follows from Proposition \ref{pUlrich}.
\end{proof}

\section{Existence of rank $2$ special Ulrich bundles}
\label{sExistence}

The existence of special rank $2$ Ulrich bundles on surfaces $S$ with $p_g(S)=q(S)=0$ and such that $\vert h-K_S\vert$ contains an irreducible curve has been proved in \cite{Bea2} by generalizing Lazarsfeld--Mukai construction. We give below a more direct construction based on the classical Hartshorne--Serre correspondence. 

\begin{lemma}
\label{lVanishing}
If $S$ is a surface with $p_g(S)=0$ and endowed with a very ample line bundle  $\cO_S(h)$, then $h^2\big(S,\cO_S(h)\big)=0$ and 
$$
h^0\big(S,\cO_S(h)\big)=h^1\big(S,\cO_S(h)\big)+\frac12(h^2-hK_S)+1-q(S).
$$
\end{lemma}
\begin{proof}
Let $H$ be a general hyperplane section of $S$, then the cohomology of the exact sequence 
\begin{equation}
\label{seqHyperplane}
0\longrightarrow\cO_S(-h)\longrightarrow\cO_S\longrightarrow\cO_H\longrightarrow0
\end{equation}
tensored by $\cO_S(h)$ implies that $h^2\big(S,\cO_S(h)\big)=0$, because $p_g(S)=0$. The second part of the statement follows immediately from Formula \eqref{RRGeneral}.
\end{proof}

From now on we will make the extra assumptions $h^1\big(S,\cO_S(h)\big)=0$ or, in other words, that $\cO_S(h)$ is non--special, and $q(S)=0$. Thus Lemma \ref{lVanishing} implies 
$$
h^0\big(S,\cO_S(h)\big)=\frac12(h^2-hK_S)+1.
$$
From now on we set
\begin{equation}
\label{eqDimension}
N:=h^0\big(S,\cO_S(h)\big)-1=\frac12(h^2-hK_S)=h^2-\pi(\cO_S(h))+1.
\end{equation}
Thus $\cO_S(h)$ induces an embedding $S\subseteq\p{N}$, hence $N\ge2$. Notice that $\pi(\cO_S(h))=0$ if and only if $\cO_S(h)$ embeds $S$ as a surface of minimal degree: thus either $S\cong\p2$ and $\cO_S(h)\cong\cO_{\p2}(\lambda)$ with $\lambda\le2$, or $S\cong\bF_e$ and $\cO_S(h)\cong\cO_{\bF_e}(\xi+bf)$ (see e.g. \cite{G--H}, Proposition at p. 525).

\begin{remark}
\label{rh^2}
We have $h^2=hK_S+2N$, hence $h^2\ge hK_S+4$.

Moreover, if $h^2\ge2$, then $N\ge3$, whence $h^2\ge hK_S+6$ (because the embedding is proper). Since each surface in $\p3$ with $p_g(S)=0$ has degree at most $3$, if $h^2\ge4$, then $N\ge4$, whence $h^2\ge hK_S+8$.
\end{remark}

The restriction map $H^0\big(\p{N},\cO_{\p{N}}(1)\big)\to H^0\big(S,\cO_S(h)\big)$ is an isomorphism. Thus $h^0\big(\p{N},\cI_{S\vert\p{N}}(1)\big)=h^1\big(\p{N},\cI_{S\vert\p{N}}(1)\big)=0$. In particular, for each $Y\subseteq S$ the cohomology of the exact sequence
$$
0\longrightarrow\cI_{S\vert\p{N}}(1)\longrightarrow \cI_{Y\vert\p{N}}(1)\longrightarrow \cI_{Y\vert S}(h)\longrightarrow0
$$
implies that
\begin{equation}
\label{restriction}
h^0\big(S,\cI_{Y\vert S}(h)\big)=h^0\big(\p{N},\cI_{Y\vert\p{N}}(1)\big).
\end{equation}

We now prove Theorem \ref{tExistence} stated in the introduction.

\medbreak
\noindent{\it Proof of Theorem \ref{tExistence}.}
Recall that by hypothesis $p_g(S)=q(S)=h^1\big(S,\cO_S(h)\big)=0$, thus $\chi(\cO_S)=1$.

Let $Z\subseteq S$ be a general set of $N+2$ points. Then the points of $Z$ can be chosen as the fundamental points of a projective frame in $\p{N}$. It follows that $h^0\big(\p{N},\cI_{Z'\vert \p{N}}(1)\big)=0$ for each subscheme $Z'\subseteq Z$  with $\deg(Z')=N+1$.

Due to Equality \eqref{restriction}, we deduce that $h^0\big(S,\cI_{Z\vert S}(h)\big)= h^0\big(S,\cI_{Z'\vert S}(h)\big)=0$, i.e. the pair $(\cO_S(h), Z)$ trivially satisfies  the  Cayley--Bacharach property. The existence of Sequence \eqref{seqUlrich} now follows by Theorem 5.1.1 of \cite{H--L}. 

Trivially $c_1(\cE)=3h+K_S$. The bundle $\cE(-h-K_S)$ has a section vanishing exactly on $Z$ by construction, thus $c_2(\cE(-h-K_S))=\deg(Z)=N+2$ whence
\begin{align*}
c_2(\cE)=\deg(Z)+2h^2+2hK_S=\frac{1}2(5h^2+3hK_S)+2=\frac{1}2(5h^2+3hK_S)+2\chi(\cO_S).
\end{align*}
Twisting Sequence \eqref{seqUlrich} by $\cO_S(-h)$ and taking into account that $p_g(S)=q(S)=0$, we obtain $h^0\big(S,\cE(-h)\big)=h^0\big(S,\cI_{Z\vert S}(h)\big)$ which is zero, as already shown above. We conclude that $\cE$ is special and Ulrich by Corollary \ref{cUlrichSpecial}.
\qed
\medbreak

As pointed out in the introduction, we do not know if the bundle $\cE$ constructed above is actually indecomposable. Anyhow, if it is decomposable we know that its direct summands are Ulrich as well. Thus the above theorem implies immediately the following corollary.

\begin{corollary}
\label{cDecomposable}
Let $S$ be a surface with $p_g(S)=q(S)=0$ and endowed with a very ample non--special line bundle $\cO_S(h)$.

Then $S$ supports at least an indecomposable Ulrich bundle of rank $r\le2$.
\end{corollary}

\begin{remark}
The cohomology of sequence
$$
0\longrightarrow \cI_{Z\vert S}\longrightarrow \cO_S\longrightarrow \cO_Z\longrightarrow 0
$$
tensored by $\cO_S(h)$, the vanishings $h^1\big(S,\cO_S(h)\big)=h^0\big(S,\cI_{Z\vert S}(h)\big)=0$ and the equalities $h^0\big(S,\cO_S(h)\big)=N+1$, $h^0\big(Z,\cO_Z(h)\big)=\deg(Z)=N+2$ imply $h^1\big(S,\cI_{Z\vert S}(h)\big)=1$, thus 
$$
\ext^1_S(\cI_{Z\vert S}(2h),\cO_S(h+K_S)\big)\cong H^1\big(S,\cI_{Z\vert S}(h)\big)^\vee\cong k.
$$
On the one hand, it follows that the bundle $\cE$ is uniquely determined by $Z$. 

On the other hand, being $\vert K_S\vert=\emptyset$, it is not easy to compute $h^0\big(S,\cI_{Z\vert S}(h-K_S)\big)$, thus we are unable to estimate the dimension of the family of bundles that we can construct in this way without further informations on the surface $S$.

Similarly, we do not know if each rank $2$ special Ulrich bundle $\cF$ on $S$ arises in the above way. Indeed, it is not immediate to prove that $h^0\big(S,\cF(-h-K_S)\big)\ne0$ because, though $\cF$ is initialized by definition, we have $\vert K_S\vert=\emptyset$.
\end{remark}

\begin{remark}
\label{rBeauville}
The existence of Ulrich bundles as above is proved in \cite{Bea2} under the extra assumption that $\vert h-K_S\vert$ contains an irreducible curve $C$ without assuming the vanishing $h^1\big(S,\cO_S(h)\big)=0$. We will show below that this condition actually implies such a vanishing, i.e. that $\cO_S(h)$ is non--special.

We showed in Lemma \ref{lVanishing} that $h^2\big(S,\cO_S(h)\big)=0$. Moreover, $h^2\big(S,\cO_S(K_S)\big)=1$ and $h^1\big(S,\cO_S(K_S)\big)=q(S)=0$. Adjunction on $S$ yields that the canonical sheaf $\omega_C$ of $C$ is $\cO_S(h)\otimes\cO_C$, hence $h^1\big(C,\cO_S(h)\otimes\cO_C\big)=1$. Thus the cohomology of
\begin{equation}
\label{seqCurve}
0\longrightarrow\cO_S(K_S-h)\longrightarrow\cO_S\longrightarrow\cO_C\longrightarrow0
\end{equation}
tensored by $\cO_S(h)$ and the above computations give $h^1\big(S,\cO_S(h)\big)=0$.

Our assumption is more general than the one in \cite{Bea2}. Indeed, let $S\cong\Bbb F_e$ where $e$ is a non--negative integer. Using the notation of Example \ref{eRational}, let $\cO_S(h)\cong\cO_S(\xi+bf)$ be very ample, so that $a\ge 1$ and $b\ge ae+1$ (see \cite{Ha2}, Corollary V.2.18). On the one hand, it is immediate to check that $\cO_S(h)$ is non--special in this case. 

On the other hand  $\vert h-K_S\vert\ne\emptyset$, but $(h-K_S)\xi=b+2-ea-e$, which is negative for each integer $ae+1\le b\le ae+e-3$. In this case, if $e\ge4$, each element in $\vert h-K_S\vert$ is reducible because it contains properly the unique element of $\vert\xi\vert$.
\end{remark}

\section{Stability of Ulrich bundles}
\label{sStability}
In this section we deal with the stability properties of the Ulrich bundles constructed in Theorem \ref{tExistence}.

To this purpose we recall some facts about the (semi)stability of Ulrich bundles. Let $\cF$ be a vector bundle on an $n$--dimensional variety $X$ endowed with a very ample line bundle $\cO_S(h)$: the slope $\mu(\cF)$ and the reduced Hilbert polynomial $p_{\cF}(t)$ are 
$$
\mu(\cF)= c_1(\cF)h^{n-1}/\rk(\cF), \qquad p_{\cF}(t)=\chi(\cF(th))/\rk(\cF).
$$
The bundle $\cF$ is $\mu$--semistable (resp. $\mu$--stable) if  $\mu(\mathcal G) \le \mu(\cF)$ (resp. $\mu(\mathcal G)< \mu(\cF)$) for all subsheaves
$\mathcal G$ with $0<\rk(\mathcal G)<\rk(\cF)$.

The bundle $\cF$ is called semistable (resp. stable) if for all $\mathcal G$ as above $p_{\mathcal G}(t) \le  p_{\cF}(t)$ (resp. $p_{\mathcal G}(t) <  p_{\cF}(t)$) for $t\gg0$. 

We have the following chain of implications
$$
\text{$\cF$ is $\mu$--stable}\ \Rightarrow\ \text{$\cF$ is stable}\ \Rightarrow\ \text{$\cF$ is semistable}\ \Rightarrow\ \text{$\cF$ is $\mu$--semistable.}
$$

\begin{theorem}
\label{tUnstable}
Let $X$ be a smooth variety endowed with a very ample line bundle $\cO_X(H)$.

If $\cF$ is an Ulrich bundle on $X$ the following assertions hold:
\begin{enumerate}
\item $\cF$ is semistable and $\mu$--semistable;
\item $\cF$ is stable if and only if it is $\mu$--stable;
\item if
\begin{equation*}
\label{seqUnstable}
0\longrightarrow\mathcal L\longrightarrow\cF\longrightarrow\mathcal M\longrightarrow0
\end{equation*}
is an exact sequence of coherent sheaves with $\cM$ torsion free and $\mu(\mathcal L)=\mu(\cF)$, then both $\mathcal L$ and $\cM$ are Ulrich bundles.
\end{enumerate}
\end{theorem}
\begin{proof}
See Theorem 2.9 of \cite{C--H2}.
\end{proof}

We are now ready to prove Theorem \ref{tStable}. 
\medbreak
\noindent{\it Proof of Theorem \ref{tStable}.}
Recall that $q(S)=p_g(S)=0$, thus $\chi(\cO_S)=1$.

Theorem \ref{tExistence} implies that for each general set $Z\subseteq S$ of $N+2$ points, there is a special rank $2$ Ulrich bundle $\cE$ with respect to $\cO_S(h)$ fitting into Sequence \eqref{seqUlrich}.

We know that if $S\subseteq \p{N}$ has minimal degree, then either $\cO_S(h)$ embeds $S$ as a rational scroll, or $S\cong\p2$ and $\cO_S(h)\cong\cO_{\p2}(\lambda)$ where $\lambda\le2$. 

Let us examine the latter case. If $\lambda=2$, then $\Omega^1_{\p2}(3)$ is the unique Ulrich bundle of rank $2$ on $S$ (see \cite{Lin}, Corollary 4.6), thus it necessarily coincides with $\cE$. Since such a surface does not support Ulrich line bundles (see Example \ref{eRational}), Theorem \ref{tUnstable} implies its stability. If $\lambda=1$, then all the Ulrich bundles on $S$ being  aCM necessarily split as sum of line bundles thanks to the Horrocks criterion, thus $\cE$ cannot be stable in this case.

If $S$ is embedded as a rational scroll, then  Corollary to Theorem B of  \cite{F--M} yields that there are no stable Ulrich bundles of rank at least $2$ on $S$ in this case (though the statement of the aforementioned corollary mentions all the surfaces of minimal degree, reading carefully its proof one immediately checks that it concerns only scrolls).

We now want to prove that the bundle $\cE$ constructed above is actually stable for a general choice of $Z$ if $S\subseteq \p{N}$ has not minimal degree. Notice that $\pi(\cO_S(h))\ge1$ in these cases (see Equality \eqref{eqDimension}), thus we will assume such a restriction from now on. Assume that $\cE$ is not stable: thanks to Theorem \ref{tUnstable} above we know the existence of an Ulrich line subbundle $\cO_S(D)\subseteq\cE$. 

Let $\cO_S(D)$ be contained in the kernel $\cK\cong\cO_S(h+K_S)$ of the map $\cE\to\cI_{Z\vert S}(2h)$ in Sequence \eqref{seqUlrich}. On the one hand we have $h^0\big(S,\cO_S(h+K_S-D)\big)\ne0$. On the other hand we know that $\cO_S(3h+K_S-D)$ is Ulrich too, hence it is initialized: in particular we have $h^0\big(S,\cO_S(2h+K_S-D)\big)=0$, a contradiction. 

We deduce that $\cO_S(D)\not\subseteq\cK$, hence the composite map $\cO_S(D)\subseteq\cE\to\cI_{Z\vert S}(2h)$ is non--zero. Thus
\begin{equation}
\label{nonVanishing}
h^0\big(S,\cI_{Z\vert S}(2h-D)\big)\ne0,
\end{equation}
hence $h^0\big(S,\cO_S(2h-D)\big)\ge 1$. Since $\pi(\cO_S(h))\ge1$, it follows that
\begin{equation}
\label{eqAnticanonical}
h^0\big(S,\cO_S(h+K_S)\big)=\chi(\cO_S(h+K_S))=\frac12(h^2+hK_S)+1=\pi(\cO_S(h))\ge1,
\end{equation}
thanks to the Kodaira vanishing theorem. The natural addition map
$$
\vert 2h-D\vert\times \vert h+K_S\vert\longrightarrow  \vert 3h+K_S-D\vert
$$
has finite fibres, because each effective divisor on $S$ can be decomposed as the sum of two effective divisors only in finitely many ways. Thus
$$
h^0\big(S,\cO_S(2h-D)\big)+h^0\big(S,\cO_S(h+K_S)\big)\le h^0\big(S,\cO_S(3h+K_S-D)\big)+1.
$$
The line bundle $\cO_S(3h+K_S-D)$ is Ulrich, hence maximally generated: it follows that $h^0\big(S,\cO_S(3h+K_S-D)\big)=h^2$. Thus Equality \eqref{eqAnticanonical} yields
\begin{align*}
h^0\big(S,\cO_S(2h-D)\big)\le h^2+1-h^0\big(S,\cO_S(h+K_S)\big)=\frac12(h^2-hK_S)=\deg(Z)-2.
\end{align*}
The inclusion $\cI_{Z\vert S}\subseteq \cO_S$ yields $h^0\big(S,\cI_{Z\vert S}(2h-D)\big)=0$ for each general choice of the scheme $Z$, contradicting Inequality \eqref{nonVanishing}. We conclude that the bundle $\cE$ is necessarily stable in this case.
\qed
\medbreak

Let $S$ be a surface with $p_g(S)=q(S)=0$ and endowed with a very ample non--special line bundle $\cO_S(h)$. We know that if $\pi(\cO_S(h))\ge1$, then the coarse moduli space $\cM_S^{s}(2;c_1,c_2)$ parameterizing isomorphism classes of stable rank $2$ bundles on $S$ with Chern classes $c_1$ and $c_2$ given by Equalities \eqref{eqUlrichSpecial} is non--empty. 
The locus $\cM_S^{s,U}(2;c_1,c_2)\subseteq \cM_S^{s}(2;c_1,c_2)$ parameterizing stable Ulrich bundles is open as pointed out in \cite{C--H2}.

\begin{proposition}
\label{pComponent}
Let $S$ be a surface with $p_g(S)=q(S)=0$ and endowed with a very ample non--special line bundle $\cO_S(h)$.

If neither $\cO_S(h)$ embeds $S$ as a rational normal scroll, nor $S\cong\p2$ and $\cO_S(h)\cong\cO_{\p2}(1)$, then the stable bundles $\cE$ constructed in Theorem \ref{tExistence} represent points of an irreducible component $\cU_S$ of dimension at least $h^2-K_S^2+5$ in $\cM_S^{s,U}(2;c_1,c_2)$.
\end{proposition}
\begin{proof}
Let $\mathcal H_S^U$ be the open subset of the Hilbert scheme $\mathcal H_S$ of $0$--dimensional subschemes of degree $N+2$ in $S$ corresponding to sets of points in general position in $\p{N}$ with respect to the embedding induced by $\cO_S(h)$. Notice that $\mathcal H_S$ is isomorphic to an open subset of the $(N+2)$--symmetric product of $S$, hence it is irreducible.

Via the construction described in Theorem \ref{tExistence} we obtain a family $\frak E\to \mathcal H_S$ of Ulrich bundles of rank $2$ with Chern classes $c_1$ and $c_2$ satisfying Equalities \eqref{eqUlrichSpecial}. Such a family is flat, because the bundles in the family fits in the same exact sequence. The property of being stable is open in a flat family (see \cite{H--L}, Proposition 2.3.1 and Corollary 1.5.11), hence the subset $\mathcal H_S^{s,U}\subseteq \mathcal H_S^U\subseteq \mathcal H_S$  of points  corresponding to stable bundles is open. 

Thanks to Theorem \ref{tStable} $\mathcal H_S^{s,U}\ne\emptyset$ if $\pi(\cO_S(h))\ge1$. Thus, in this case, we have a morphism $\mathcal H_S^{s,U}\to \cM_S^{s,U}(2;c_1,c_2)$ whose image parameterizes the isomorphism classes of stable bundles constructed in Theorem \ref{tExistence}. In particular such bundles, correspond to points of an irreducible component $\cU_S\subseteq\cM_S^{s,U}(2;c_1,c_2)$.

As pointed out in Theorems 4.5.4 and 4.5.8 of \cite{H--L}, $\dim(\cU_S)\ge4c_2-c_1^2-3\chi(\cO_S)$.
Taking into account the definition of $c_1$ and $c_2$ above, one obtains the dimension of $\cU_S$. 
\end{proof}

The importance of the following proposition is clear. 

\begin{proposition}
\label{pSmoothModuli}
Let $S$ be a surface with $p_g(S)=q(S)=0$ and endowed with a very ample non--special line bundle $\cO_S(h)$.

Then each bundle $\cE$ constructed in Theorem \ref{tExistence} satisfies
$$
h^2\big(S,\cE\otimes\cE^\vee\big)=h^0\big(S,\cO_S(2K_S-h)\big).
$$
In particular, if $h^0\big(S,\cO_S(2K_S-h)\big)=0$ and neither $\cO_S(h)$ embeds $S$ as a rational normal scroll, nor $S\cong\p2$ and $\cO_S(h)\cong\cO_{\p2}(1)$, then $\cU_S$ has dimension $h^2-K_S^2+5$ and it is generically smooth. 
\end{proposition}
\begin{proof}
Tensoring Sequence \eqref{seqUlrich} by $\cO_S(K_S-2h)$ we obtain
$$
0\longrightarrow\cO_S(2K_S-h)\longrightarrow\cE(K_S-2h)\longrightarrow\cI_{Z\vert S}(K_S)\longrightarrow0.
$$
We have $h^0\big(S,\cI_{Z\vert S}(K_S)\big)\le h^0\big(S,\cO_S(K_S)\big)=p_g(S)=0$. The cohomology of the above sequence then yields $h^0\big(S,\cE(K_S-2h)\big)=h^0\big(S,\cO_S(2K_S-h)\big)$. 

Tensoring Sequence \eqref{seqUlrich} by $\cE^\vee(K_S)\cong\cE(-3h)$ we also obtain the exact sequence 
$$
0\longrightarrow\cE(K_S-2h)\longrightarrow\cE\otimes\cE^\vee(K_S)\longrightarrow\cI_{Z\vert S}\otimes\cE(-h)\longrightarrow0.
$$
Trivially $h^0\big(S,\cI_{Z\vert S}\otimes\cE(-h)\big)\le h^0\big(S,\cE(-h)\big)=0$. Thus 
$$
h^2\big(S,\cE\otimes\cE^\vee\big)=h^0\big(S,\cE\otimes\cE^\vee(K_S)\big)=h^0\big(S,\cE(K_S-2h)\big)=h^0\big(S,\cO_S(2K_S-h)\big).
$$

Now assume that $\pi(\cO_S(h))\ge1$ and $h^0\big(S,\cO_S(2K_S-h)\big)=0$. Theorem \ref{tStable} implies that the general $\cE$ as above corresponds to a point in $\cU_S\subseteq\cM_S^{s,U}(2;c_1,c_2)$. Such a point is smooth and $\dim(\cU_S)=4c_2-c_1^2-3\chi(\cO_S)=h^2-K_S^2+5$, thanks to Theorem 4.5.8 of \cite{H--L}. 
\end{proof}

\begin{example}
\label{eEnriques}
Let $S$ be an {\sl Enriques surface}, i.e. a surface with $q(S)=p_g(S)=0$, trivial bicanonical class and which is also {\sl minimal}, i.e. it does not contain curves $E\cong\p1$ with $E^2=-1$. 

Since $\cO_S(h-K_S)$ is ample, it follows from the Kodaira vanishing theorem that  $h^1\big(S,\cO_S(h)\big)=0$. Moreover $2\pi(\cO_S(h))=h^2+2\ge2$. Theorems \ref{tExistence} and \ref{tStable} yields the existence of a stable special Ulrich bundle $\cE$ of rank $2$ on $S$. 

Notice that $h^0\big(S,\cO_S(2K_S-h)\big)=h^0\big(S,\cO_S(-h)\big)=0$ in this case, hence Propositions \ref{pComponent} and \ref{pSmoothModuli} imply that $\cU_S$ is generically smooth of dimension $h^2+5$.
\end{example}

\begin{remark}
\label{rSmoothModuli}
We now examine Beauville's hypothesis from the stability viewpoint. If $h^0\big(S,\cO_S(h-K_S)\big)\ne0$, then the cohomology of Sequence \eqref{seqCurve} tensored by $\cO_S(K_S)$ yields $h^0\big(S,\cO_S(2K_S-h)\big)\le p_g(S)=0$ and Proposition \ref{pSmoothModuli} holds in this case. 

If $h^0\big(S,\cO_S(h-K_S)\big)=0$, then the cohomology of Sequence \eqref{seqUlrich} tensored by $\cO_S(-h-K_S)$ implies that 
\begin{align*}
1=h^0\big(S,\cO_S\big)&\le h^0\big(S,\cE(-h-K_S)\big)= \\
&=1+h^0\big(S,\cI_{Z\vert S}(h-K_S)\big)\le 1+h^0\big(S,\cO_S(h-K_S)\big)=1.
\end{align*}
It follows that the map $\mathcal H_S^{s,U}\to\cM^{s,U}_S(2;c_1,c_2)$ is  injective when $\pi(\cO_S(h))\ge1$, hence 
$\dim(\cU_S)\ge2(N+2)=h^2-hK_S+4$.
\end{remark}

\section{Ulrich--wildness}
\label{sWild}
In this section we will deal with the Ulrich--wildness of the surfaces endowed with a very ample non--special line bundle.
We will make use of the following result.

\begin{theorem}
\label{tFPL}
Let $X$ be a smooth variety endowed with a very ample line bundle $\cO_X(h)$.

If $\cA$ and $\cB$ are simple Ulrich bundles on $X$ such that $h^1\big(X,\cA\otimes\cB^\vee\big)\ge3$ and $h^0\big(X,\cA\otimes\cB^\vee\big)=h^0\big(X,\cB\otimes\cA^\vee\big)=0$, then $X$ is Ulrich--wild.
\end{theorem}
\begin{proof}
See \cite{F--PL}, Theorem 1 and Corollary 1.
\end{proof}

An immediate consequence of the above theorem is the following easy result.

\begin{lemma}
\label{lWild}
Let $S$ be a surface with $p_g(S)=q(S)=0$ and endowed with a very ample non--special line bundle $\cO_S(h)$.

If $\pi(\cO_S(h))\ge1$ and $h^2+1\ge K_S^2$, then $S$ is Ulrich--wild.
\end{lemma}
\begin{proof}
Thanks to Theorems \ref{tExistence}, \ref{tStable}, Proposition \ref{pComponent} and the hypothesis on $h^2$ we already know the existence of an irreducible component $\cU_S\subseteq\cM_S^{s}(2;c_1,c_2)$ of dimension at least $h^2-K_S^2+5\ge4$, whose points correspond to special stable Ulrich bundles of rank $2$. 

Let $\cE$ and $\cG$ bundles corresponding to distinct points in $\cU_S$. Thanks to \cite{H--L}, Corollary 1.2.8, both $\cE$ and $\cG$, being stable, are simple. Proposition 1.2.7 of \cite{H--L} implies $h^0\big(F,\cE\otimes\cG^\vee\big)=h^0\big(F,\cG\otimes\cE^\vee\big)=0$, thus
$$
h^1\big(F,\cE\otimes\cG^\vee\big)=h^2\big(F,\cE\otimes\cG^\vee\big)-\chi(\cE\otimes\cG^\vee)\ge -\chi(\cE\otimes\cG^\vee).
$$
Equality \eqref{RRGeneral} with $\cF:=\cE\otimes\cG^\vee$ and the equalities $\rk(\cE\otimes\cG^\vee)=4$, $c_1(\cE\otimes\cG^\vee)=0$ and $c_2(\cE\otimes\cG^\vee)=4c_2-c_1^2$ imply
$$
h^1\big(F,\cE\otimes\cG^\vee\big)\ge 4c_2-c_1^2-4\chi(\cO_S)=h^2-K_S^2+4\ge3.
$$ 
We conclude that $S$ is Ulrich--wild, by Theorem \ref{tFPL}.
\end{proof}

\medbreak
\noindent{\it Proof of Theorem \ref{tWild}.}
Let $S$ be a surface with $p_g(S)=q(S)=0$ and assume first that $\pi(\cO_S(h))=0$. In this case $S\subseteq \p{N}$ has minimal degree (see Equality \eqref{eqDimension}). Using the notation of Example \ref{eRational}, either $S\cong\p2$ and $\cO_S(h)\cong\cO_{\p2}(\lambda)$ with $\lambda\le2$, or $S\cong\bF_e$ and $\cO_S(h)\cong\cO_{\bF_e}(\xi+bf)$ (see e.g. \cite{G--H}, Proposition at p. 525). 

In the first case $h^2\le4$ and $S$ supports a finite number of Ulrich bundles (see \cite{E--He}). In the second case such surfaces are Ulrich--wild if and only if $h^2\ge5$ thanks to \cite{MR} (see also \cite{F--M} for some comments on the proof in \cite{MR}).

If $\pi(\cO_S(h))\ge1$ and $h^2\le7$, then the list given in \cite{Io2} and the hypothesis $p_g(S)=q(S)=0$ yield that  $h^2\ge K_S^2$. Thus the statement follows immediately from Lemma \ref{lWild}. 

Assume $\pi(\cO_S(h))\ge1$ and $h^2\ge8$. Recall that $S$ is the blow up of a minimal surface $\Sigma$ with $p_g(\Sigma)=q(\Sigma)=0$. The classification of such surfaces  is part of the more general Enriques--Kodaira classification (see Chapter VI of \cite{B--H--P--VV}). If $\kappa(\Sigma)$ is the Kodaira dimension of $\Sigma$, we have the following possible cases for the aforementioned surfaces.
\begin{itemize}
\item If $\kappa(\Sigma)=-\infty$, then $\Sigma$ is either $\p2$, or $\Bbb F_e:=\Bbb P(\cO_{\p1}\oplus\cO_{\p1}(-e))$ for some integer $e\ge0$, $e\ne1$. In this case $K_{\Sigma}^2=9,8$ according $\Sigma\cong\p2$ or not.
\item If $0\le \kappa(\Sigma)\le1$, then $\Sigma$ is elliptic. In this case $K_{\Sigma}^2=0$.
\item If $\kappa(\Sigma)=2$, then $\Sigma$ is of general type. In this case $1\le K_{\Sigma}^2\le9$ (see \cite{B--H--P--VV}, Section VII.10).
\end{itemize}

We deduce that $K_S^2\le9$ in this case, hence the statement is again an immediate consequence of Lemma \ref{lWild}.
\qed
\medbreak

\begin{corollary}
Let $S\subseteq\p N$ be an Enriques surface. Then $S$ is Ulrich--wild.
\end{corollary}
\begin{proof}
As already pointed out in Example \ref{eEnriques}, $\cO_S(h)$ is non--special. Moreover, $h^2\ge8$ by Remark \ref{rh^2}, because $hK_S=0$.
\end{proof}

An analogous result can be found in \cite{B--N} when $\cO_S(h)$ is a multiple of the Fano polarisation (see the aforementioned paper for the definition and details about such a polarisation) and $S$ is {\sl unnodal}, i.e. $S$ does not contain curves $E\cong\p1$ such that $E^2=-2$. 

Notice that a linearly normal Enriques surface $S\subseteq\p N$ is not necessarily aCM (see \cite{G--L--M}).

\begin{example}
\label{eAlexander}
Let $S\subseteq\p4$ be a non--special surface with $p_g(S)=q(S)=0$ and {\sl non--degenerate}, i.e. not contained in any hyperplane. Then the embedding is induced by a monomorphism $H^0\big(\p4,\cO_{\p4}(1)\big)\to H^0\big(S,\cO_{S}(h)\big)$. 

The Veronese surface in $\p4$, i.e. the projection from a general point of $\p5$ of the image of $\p2$ via the map defined by $\cO_{\p2}(2)$ is the unique smooth surface in $\p4$ which is not linearly normal (see \cite{Se}).

In \cite{Al}, Th\'eor\ga eme (1), the author gives the complete list of non--degenerate non--special linearly normal surfaces when they are rational. Indeed $S$ is the blow up $\Bl_X\p2$ of $\p2$ at a suitable set of distinct points  $X:=\{\ Q_1,\dots,Q_n,P_1,\dots,P_m\ \}$. Let $\pi\colon \Bl_X\p2\to\p2$ be the canonical projection, and denote by $\ell$ the pull back to $\Bl_X\p2$ of the class of a general line in $\p2$ and by $f_i:=\pi^{-1}(Q_i)$, $e_j:=\pi^{-1}(P_j)$ the exceptional divisor on $\Bl_X\p2$. Then the linearly normal non--degenerate non--special rational surface in $\p4$ are exactly the ones in Table 1 below.

The cohomology of Sequence \eqref{seqHyperplane} tensored by $\cO_S(h)$ implies that non--special surfaces $S\subseteq \p4$ with $p_g(S)=0$ are also {\sl sectionally non--special}, i.e. their general hyperplane section $H$ satisfies $h^1\big(H,\cO_H\otimes\cO_S(h)\big)=0$.
Thanks to the results listed in \cite{I--M} and \cite{M--R} we know that $S$ is:
\begin{itemize}
\item either the Veronese surface in $\p4$;
\item or one of the rational surfaces listed in Table 1;
\item or $d=9$ and there is a blow up morphism $\pi\colon S\to \Sigma$ at a point $P$ of an Enriques surface $\Sigma$, and $\cO_S(h)\cong\cO_S(\pi^*L-\pi^{-1}(P))$ for a suitable very ample line bundle $\cO_\Sigma(L)$ with $L^2=10$: in this case $K_S^2=-1$ and $h^2=9$.
\end{itemize}

In all the cases  $S$ supports special Ulrich bundles of rank $2$ by Theorem \ref{tStable}. Moreover, $S$ supports stable special Ulrich bundles of rank $2$ if and only if $h^2\ge4$. Similarly $S$ is Ulrich--wild if and only if $h^2\ge4$ and it is not the Veronese surface.

When $4\le h^2\le 6$ and $S$ is not the Veronese surface, the existence of stable special Ulrich bundles of rank $2$ on $S$ and its Ulrich--wildness  were proved in \cite{MR--PL1} as part of a more general result: indeed this is exactly the case when the embedded surface $S$ is aCM (see the aforementioned paper and the references therein for the details).  When $d\ge7$, the surface $S$ is no more aCM.

\begin{gather*}
\text{Linearly normal non--degenerate non--special rational surface in $\p4$.}\\
\begin{array}{|c|c|c|}           \hline
X&  h&K_S^2   \\ \hline
\{\ P_1\ \}&2\ell-e_1&8\\  \hline
\{\ P_1,\dots,P_5\ \}&3\ell-\sum_{j=1}^4e_j&4\\  \hline
\{\ Q_1, P_1,\dots,P_7\ \}&4\ell-2f_1-\sum_{j=1}^7e_j&1\\  \hline
\{\  P_1,\dots,P_{10}\ \}&4\ell-\sum_{i=1}^{10}e_i&-1\\  \hline
\{\  Q_1,\dots,Q_6,P_1,\dots,P_5\ \}&6\ell-2\sum_{i=1}^6f_i-\sum_{j=1}^{5}e_j&-2\\  \hline
\{\  Q_1,\dots,Q_{10},P_1\ \}&7\ell-2\sum_{i=1}^{10}f_i-e_1&-2\\  \hline
\{\  Q_1,\dots,Q_{10}\ \}&13\ell-4\sum_{i=1}^{10}f_j&-1\\  \hline
\end{array}
\end{gather*}
\end{example}

\section{Anticanonical rational surfaces}
\label{sAnticanonical}
In this section we will deal with the anticanonical surfaces. E.g. each Hirzebruch surface $\bF_e$ or each blow up of $\p2$ at any set of points lying on a cubic curve: in particular, each linearly normal non--special surface of degree up to $5$ in $\p4$ is anticanonical (see Example \ref{eAlexander}). 

\medbreak
\noindent{\it Proof of Theorem \ref{tAnticanonical}.}
Let $S$ be an anticanonical rational surface: we have that $q(S)=p_g(S)=0$. If $A\in\vert -K_S\vert$, then $\omega_A\cong\cO_A$ by the adjunction formula, hence $h^1\big(A,\cO_S(h)\otimes\cO_A\big)=h^0\big(A,\cO_S(-h)\otimes\cO_A\big)$. 

On the one hand, if $h^0\big(A,\cO_S(-h)\otimes\cO_A\big)>0$, then $h^0\big(C,\cO_S(-h)\otimes\cO_C\big)>0$ for some irreducible component $C\subseteq A$, hence $-hC\ge0$. On the other hand we know that $-hC<0$, because $\cO_S(h)$ is ample. 

We deduce that $h^0\big(A,\cO_S(-h)\otimes\cO_A\big)=0$, hence the cohomology of the exact sequence
$$
0\longrightarrow \cO_S(h+K_S)\longrightarrow \cO_S(h)\longrightarrow \cO_S(h)\otimes\cO_A\longrightarrow0
$$
and the Kodaira vanishing theorem yield $h^1\big(S,\cO_S(h)\big)=h^2\big(S,\cO_S(h)\big)=0$, i.e. $\cO_S(h)$ is non--special. It follows from Theorem \ref{tExistence} that $S$ supports special Ulrich bundles of rank $2$. Theorem \ref{tStable} implies that such bundles are also stable if $\pi(\cO_S(h))\ge1$. The assertion about their moduli space follows from Proposition 3.11 of \cite{C--MR}. 
\qed
\medbreak

As an immediate consequence of the above Theorem one obtains the following result.

\begin{corollary}
\label{cAnticanonicalWild}
Let $S$ be an anticanonical rational surface endowed with a very ample line bundle $\cO_S(h)$. Then $S$ is Ulrich--wild if and only if $h^2\ge4$.
\end{corollary}
\begin{proof}
The line bundle $\cO_S(h)$ is non--special, hence the statement follows from Theorem \ref{tStable}.
\end{proof}

The above Theorem \ref{tAnticanonical} is a generalisation of Theorem 1 of \cite{Kim}. In order to understand the results proved therein, we recall some definitions. The {\sl Clifford index of a line bundle $\mathcal L$} on a smooth curve $C$ is
$$
\Cliff(\mathcal L):=\deg(\mathcal L)-2h^0\big(C,\mathcal L\big)+2.
$$
The {\sl Clifford index of $C$} is 
$$
\Cliff(C):=\min\{\ \Cliff(\mathcal L)\ \vert\ \mathcal L\in\Pic(C), h^0\big(C,\mathcal L\big)\ge2,\ h^1\big(C,\mathcal L\big)\ge2\ \}
$$
Finally, $\mathcal L\in\Pic(C)$ {\sl computes the Clifford index} of $C$ if $\Cliff(C)=\Cliff(\mathcal L)$.

Using a construction due to Lazarsfeld and Mukai, the author proves in Theorem 1 of \cite{Kim} that for each curve $C\in\vert 3h+K_S\vert$ such that $\Cliff(C)$ is computed by $\cO_S(h+K_S)\otimes\cO_C$, then $S$ supports an irreducible family of dimension $h^2-K_S^2+5$  of rank $2$ stable special Ulrich bundles. 

Notice that if $\pi(\cO_S(h))=0$, then $\cO_S(h+K_S)\otimes\cO_C$ has degree
$$
3h^2+4hK_S+K_S^2=-h^2-8+K_S^2.
$$
The Enriques--Kodaira classification of minimal surfaces $S$ with $p_g(S)=q(S)=0$ (see the proof of Theorem \ref{tWild}) implies $K_S^2\le9$. Thus $-h^2-8+K_S^2\le 1-h^2\le 0$, i.e. $S\cong\p2$ and $\cO_S(h)\cong\cO_{\p2}(1)$. Since in this case $\vert 3h+K_S\vert=\emptyset$, it follows that the hypothesis in \cite{Kim} on $\Cliff(C)$ forces $\pi(\cO_S(h))\ge1$.

Thus the existence of  stable special Ulrich bundles proved under the technical hypothesis of Theorem 1 of \cite{Kim}, actually follows soon from Theorem \ref{tAnticanonical} which improves considerably the result in \cite{Kim} as the following two examples show.

\begin{example}
The following example extends Corollary 1 and Remark 6 of \cite{Kim}.

Consider a del Pezzo surface of degree $d\ge3$ which is the blow up of $\p2$ at $m=9-d\le 6$ general points. The linear system $\cO_S(h):=\cO_S(-K_S)$ is very ample and $\pi(\cO_S(h))=1$, hence the anticanonical embedding of $S$ supports stable special Ulrich bundles of rank $2$. Nevertheless, $\cO_S(h+K_S)\cong\cO_S$ does not compute the $\Cliff(C)$ for each $C\in\vert 3h+K_S\vert$. In particular Theorem \ref{tExistence} actually extends Theorem 1 of \cite{Kim}.

As pointed out in the introduction, it is well known that each del Pezzo surface with $d\ge3$ supports special Ulrich bundles of rank $2$ and is Ulrich--wild is well--known: see \cite{E--S--W}, Corollary 6.5 for the existence and \cite{PL--T} and \cite{MR--PL2} for the Ulrich--wildness. Notice that del Pezzo surfaces with $8\ge d\ge3$ also support many stable Ulrich bundles of rank $2$ which are not special (see \cite{Cs} for their classification).
\end{example}

\begin{example}
Let $S$ be the blow--up of $\p 2$ at $m\ge2$ general points on a cubic curve and let $p\colon S\to\p2$ be the blow up map. The surface $S$ is trivially anticanonical by definition. Moreover, thanks to Proposition at p.525 of \cite{G--H} its image in $\p{N}$ has not minimal degree, hence for each very ample line bundle $\cO_S(h)$ on $S$ one necessarily has $\pi(\cO_S(h))\ge1$ (see Equality \eqref{eqDimension}). Thus such an $S$ supports stable special Ulrich bundles of rank $2$. 

The Picard group of $S$ is freely generated by the pull--back $\ell$ via $p$ of a general line and by the exceptional  divisors $e_1,\dots,e_m$ of $\pi$. 
Let $h:=a\ell-\sum_{i=1}^me_i$ with $m\le9$. \cite{DA--H}, Th\'eor\ga eme 2.3 implies that $\cO_S(h)$ is very ample when $a\ge5$. When $a=4$ the same result holds, because the linear system induced on the blow up at $10$ points by quartics through those points is very ample (e.g. see \cite{Al}). We have
$$
\vert 3h+K_S\vert=\vert3(a-1)\ell-2\sum_{i=1}^me_i\vert,\qquad \vert h+K_S\vert=\vert(a-3)\ell\vert.
$$
Then $\pi(\cO_S(h))$ and the genus $g$ of a smooth $C\in\vert 3h+K_S\vert$ satisfy
$$
\pi(\cO_S(h))={{a-1}\choose2}\ge1,\qquad g=\frac{(3a-4)(3a-5)}2-m\ge4
$$
for $a\ge4$ and $m\le 9$. In particular, both $3a-5$, and $(h+K_S)C=3(a-1)(a-3)$ are strictly smaller than $g-1$ for $a\ge4$ and $m\le 9$. 

On the one hand $C$ has a $g^1_{3a-5}$ cut out by the lines through any of the blown up points. Thus the $\Cliff(C)\le 3a-7$. On the other hand, $p_*\cO_S\cong\cO_{\p2}$ and $R^ip_*\cO_S=0$ for $i\ge1$, hence the projection formula yields $h^i\big(S,\cO_S(t\ell)\big)=h^i\big(\p2,\cO_{\p2}(t)\big)$ for each integer $t$. 

We trivially have $h^0\big(S,\cO_S(h+K_S-C)\big)=0$. Moreover, we also have $h^1\big(S,\cO_S(h+K_S-C)\big)=h^1\big(S,\cO_S(C-h)\big)=0$ because $\vert C-h\vert$ is non--special for the general choice of the blown up points.

It follows from the cohomology of sequence
$$
0\longrightarrow\cO_S(h+K_S-C)\longrightarrow\cO_S(h+K_S)\longrightarrow\cO_S(h+K_S)\otimes\cO_C\longrightarrow0
$$
that the $\Cliff(\cO_S(h+K_S)\otimes\cO_C)=(2a-3)(a-3)$ which is strictly greater than $3a-7$ when $a\ge5$ (and coincides with it if $a=4$). We conclude that the polarised surface $S$ does not satisfy the hypothesis of Theorem 1 of \cite{Kim} when $a\ge5$.
\end{example}

\section{Acknowledgments}
The author is a member of GNSAGA group of INdAM and is supported by the framework of PRIN 2015 \lq Geometry of Algebraic Varieties\rq, cofinanced by MIUR.

The author is particularly indebted with the referee for her/his criticisms, questions, remarks and suggestions which have considerably improved the whole exposition. Finally, the author also thanks the editor of the International Journal of Mathematics, Professor Oscar Garcia--Prada.

\bigskip
\noindent
Gianfranco Casnati,\\
Dipartimento di Scienze Matematiche, Politecnico di Torino,\\
c.so Duca degli Abruzzi 24, 10129 Torino, Italy\\
e-mail: {\tt gianfranco.casnati@polito.it}

\end{document}